\documentclass[12pt,UTF8]{article}
\usepackage{amssymb,latexsym}

\usepackage{color}

\usepackage[colorlinks, backref=page
]{hyperref}
\usepackage{amsmath,amsthm}
\usepackage{indentfirst}
\theoremstyle{plain}
\newtheorem{thm}{Theorem}

\newtheorem{lem}{Lemma}

\theoremstyle{definition}

\newtheorem{rem}{Remark}
\numberwithin{equation}{section}
\allowdisplaybreaks

\title{Eells-Sampson type result of Symphonic maps}

\author{Xiang-Zhi Cao\thanks{School of Information Engineering, Nanjing Xiaozhuang University, Nanjing 211171, China}\,\,\thanks{Supported by General Project of Natural Science (Basic Science) Research of Colleges and
		Universities in Jiangsu Province (No. 22KJD110004)}}

\newcommand{\cnabstractname}{摘要}

\newcommand{\enabstractname}{Abstract}

\begin{document}
	
	\maketitle
	
	\begin{abstract}
		In this paper, we  obtained Eells-Sampson type result of Symphonic maps.
	\end{abstract}
		\section{Introduction}\label{ddd}

	Let $ f:(M,g)\to (N,h) $ be a smooth map between Riemannian manifold. Consider the following symphonic energy of a  smooth map
	$$
	\begin{aligned}
		E_{sym}(f)=\int_M \frac{\|f^* h \|^2}{2} dv_g.
	\end{aligned}
	$$
	
	Let $ \{e_i\}_{i=1}^{m} $ be the local orthonormal frame.    
	Its Euler-Lagrange equation of $E_{sym}(f)$ is as follows (cf. \cite{MR2781758,MR2852338}) 
	\begin{equation}\label{cg}
		\begin{aligned}
			\operatorname{div}_g \sigma_f=0.
		\end{aligned}
	\end{equation}
	where $\sigma_f(X)=\sum_j h\left(d f(X), d f\left(e_j\right)\right) d f\left(e_j\right) $.
	
	In 2011, Nakauchi and Takenaka \cite{MR2852338} defined symphonic map, which is defined as the critical point of the above funtional. One can refer to  \cite{MR3940323,MR3813997,MR3451407,MR4483578, MR4434687,MR3989749,MR4484437} for the related results for symphonic map.

%

		\begin{thm}[cf.\cite{MR4756031}]
		Let $f:\left(M^{m \geq 2}, g\right) \rightarrow(N, h)$ be a harmonic map between Riemannian manifolds. Assume that $M$ is closed and that there exists $K>0$ such that
		
		$$
		\operatorname{Ric}_g \geq(m-1) K f^* h \quad \text { and } \quad \sec _h \leq K
		$$

		Then $f$ is a totally geodesic map. In particular, either
		i) $f$ is constant, or
		ii) $f$ is a homothetic immersion, $g$ has positive constant curvature and  on $M$ and $f(M)$, respectively
		
		$$
		\operatorname{Ric}_g=(m-1) K f^* h \quad \text { and } \quad \sec _h(\Pi)=K
		$$		
		for any 2-plane $\Pi$ contained in $\mathrm{d} f(T M) \subseteq T N$.

	\end{thm}

	About symphonic map, Nakauchi  has obtained the similar theorems:

	\begin{thm}[Compact case,\cite{MR3940323}]
		Let $(M,g)$ be a compact connected Riemannian manifold without boundary, and let $(N,h)$ be a Riemannian manifold. Suppose that the Ricci curvature of $M$ is non-negative, and that the sectional curvature of $N$ is non-positive. Let $f$ be any symphonic map of class $\mathrm{C}^3$ from $M$ into $N$.

		Then
		
		(a) The pullback $f^* h$ is parallel, i.e., $\nabla\left(f^* h\right)=0$.
		
		(b) If the Ricci curvature of $M$ is positive at a point, then $f$ is a constant map.
		(c) If $N$ has negative sectional curvature everywhere, then either $f$ is a constant map or $d f(x)$ has at most rank one for all $x \in M$.
	\end{thm}

	The main purpose of this paper   are to generalize this  theorem  to weak curvature condition.
	
		\begin{thm}\label{thm1} Let $(M^m,g)$ be a compact connected Riemannian manifold without boundary, and let $(N,h)$ be a Riemannian manifold with curvature tensor 
		$$R^N(X,Y,Z,W)=a(h(X,Z)h(Y,W))-h(X,W)h(Y,Z).$$ Let $f$ be any symphonic map of class $\mathrm{C}^3$ from $M$ into $N$. Suppose that 
		$$
		\operatorname{Ric}_g \geq(m-1) K f^* h \quad \text { and } \quad a \leq K.
		$$

		Then the pullback $f^* h$ is parallel, i.e., $\nabla\left(f^* h\right)=0$. Moreover,  $f$ is constant map.

%
		
		
	\end{thm}
	
	\begin{rem}
		In \cite{MR4483578},  The author  got the similar theorem, but the conditions is different here, although we both weaken the upper bound of sectional curvature of the target manifold. 
	\end{rem}
	
	Using the argument in the proof of Theorem 3, we immediately  get
	
	\begin{thm}[Noncompact case ]
	Let $M^m$ be a complete connected non-compact Riemannian manifold  and let $N$ be a Riemannian manifold with curvature tensor 
	$$R^N(X,Y,Z,W)=a(h(X,Z)h(Y,W))-h(X,W)h(Y,Z).$$ Suppose that 
	$$
	\operatorname{Ric}_g \geq(m-1) K f^* h \quad \text { and } \quad a \leq K
	$$ 
	If  $E_{\text {sym }}(f)<\infty$, then $f$ is a constant map.
\end{thm}
\begin{proof}
	From the proof of Theorem \ref{thm1}, we know 
	$$\begin{aligned}
		&\frac{1}{4} \Delta\left\|f^* h\right\|^2 \geq \operatorname{div}_g \alpha_f+\frac{1}{2}\left\|\nabla\left(f^* h\right)\right\|^2 .
	\end{aligned}$$
	the rest proof proceeds as in \cite{MR3940323}.
\end{proof}
	
	Acknowledgement: This work is Supported by General Project of Natural Science (Basic Science) Research of Colleges and Universities in Jiangsu Province (No. 22KJD110004)
	
	In this paper, we use notation: $R^N(X,Y,Z,W)=\langle R(X,Y)W, Z\rangle$
	
	\section{the proof of main theorem }

	\begin{lem}\label{lem1}
		(1) If symmetric matrices $A=\left(a_{i j}\right)$ and $B=\left(b_{i j}\right)$ are positive semi-definite,		
	then we have
		
		$$
		\operatorname{tr}\left(A^t B\right)=\sum_{i, j=1}^m a_{i j} b_{i j} \geq 0
		$$
		
		(2) If $A$ is positive definite and $B$ is positive semi-definite, then $\operatorname{tr}\left(A^t B\right)=0$ implies $B=0$.
	\end{lem}
	
	\begin{lem}[Formula of Bochner type, cf.\cite{MR3940323}]
	 For any map $f$ of class $\mathrm{C}^3$ from $M$ into $N$, the following equality holds:
		
		$$
		\begin{aligned}
			& \frac{1}{4} \Delta\left\|f^* h\right\|^2 \\
			& =\operatorname{div}_g \alpha_f-h\left(\tau_f, \operatorname{div}_g \sigma_f\right) . \\
			& \quad+\sum_{i, j, k=1}^m h\left(\left(\nabla_{e_k} d f\right)\left(e_i\right),\left(\nabla_{e_k} d f\right)\left(e_j\right)\right) h\left(d f\left(e_i\right), d f\left(e_j\right)\right) \\
			& \quad+\frac{1}{2}\left\|\nabla\left(f^* h\right)\right\|^2 \\
			& \quad+\sum_{i, j, k=1}^m h\left(d f\left(Ric^M \left(e_i, e_k\right)\left(e_k\right)\right), d f\left(e_j\right)\right) h\left(d f\left(e_i\right), d f\left(e_j\right)\right) \\
			& \quad-\sum_{i, j, k=1}^m h\left( R^N\left(d f\left(e_i\right), d f\left(e_k\right)\right) d f\left(e_k\right), d f\left(e_j\right)\right) h\left(d f\left(e_i\right), d f\left(e_j\right)\right),
		\end{aligned}
		$$		
		where $\left\{e_i\right\}$ denotes a local orthonormal frame around the point of interest in $M$.
	\end{lem}

	\begin{proof}[Proof of Theorem \ref{thm1}]
%
%
%

		Let
		
		$$
		\begin{aligned}
			a_{i k} & =R_{ik}-(m-1)Kb_{ik} \\
			c_{i k} & =\sum_{j=1}^m h\left(d f\left(e_k\right), d f\left(e_j\right)\right) h\left(d f\left(e_i\right), d f\left(e_j\right)\right) .
		\end{aligned}
		$$
		
		$$
		\begin{aligned}
			& d_{i j}=-\sum_{k=1}^m h\left( R^N\left(d f\left(e_i\right), d f\left(e_k\right)\right) d f\left(e_k\right), d f\left(e_j\right)\right) \\
			& b_{i j}=h\left(d f\left(e_i\right), d f\left(e_j\right)\right)
		\end{aligned}
		$$
		
It is easy  to see that  the symmetric matrix $\left(a_{i j}\right)$ is positive semi-definite.

Notice that
		
$$\begin{aligned}
&\sum_{i, k=1}^m R_{i k} \sum_{j=1}^m h\left(d f\left(e_k\right), d f\left(e_j\right)\right) h\left(d f\left(e_i\right), d f\left(e_j\right)\right) 	\\
&-(m-1)K\sum_{k=1}^m h\left(\left(d f\left(e_i\right), d f\left(e_k\right)\right) d f\left(e_k\right), d f\left(e_j\right)\right)b_{ij}\\
=&R_{ik}b_{kj}b_{ij}-(m-1)Kb_{ik}b_{kj}b_{ij}\\
=&(R_{ik}-(m-1)Kb_{ik})b_{kj}b_{ij}\\
&\geq 0. 
\end{aligned}$$	
and

$$
\begin{aligned}
 & (m-1)K\sum_{k=1}^m h\left(\left(d f\left(e_i\right), d f\left(e_k\right)\right) d f\left(e_k\right), d f\left(e_j\right)\right)b_{ij}\\
 &-\sum_{i, j, k=1}^m h\left( R^N\left(d f\left(e_i\right), d f\left(e_k\right)\right) d f\left(e_k\right), d f\left(e_j\right)\right) b_{ij} \\
	 =&(m-1)Kb_{ik}b_{kj}b_{ij}-\sum_{i, j, k=1}^m h\left( R^N\left(d f\left(e_i\right), d f\left(e_k\right)\right) d f\left(e_k\right), d f\left(e_j\right)\right) b_{ij} \\
	=&(m-1)Kb_{ik}b_{kj}b_{ij}-a(b_{ik}b_{kj}-b_{ij}b_{kk})b_{ij}\\
	=&((m-1)K-a)b_{ik}b_{kj}b_{ij}+b_{ij}^2b_{kk}
\end{aligned}
$$

	By divergence  theorem on compact manifold, we have 
	
	\begin{equation}
		\begin{split}
		0	& \geq \int_M \quad\sum_{i, j, k=1}^m h\left(\left(\nabla_{e_k} d f\right)\left(e_i\right),\left(\nabla_{e_k} d f\right)\left(e_j\right)\right) h\left(d f\left(e_i\right), d f\left(e_j\right)\right) \\
			& \quad+\int_M\frac{1}{2}\left\|\nabla\left(f^* h\right)\right\|^2 \\
			& \quad+\int_M\sum_{i, j, k=1}^m h\left(d f\left( Ric^M\left(e_i, e_k\right)\left(e_k\right)\right), d f\left(e_j\right)\right) h\left(d f\left(e_i\right), d f\left(e_j\right)\right)\\
			&- (m-1)K\sum_{k=1}^m h\left(\left(d f\left(e_i\right), d f\left(e_k\right)\right) d f\left(e_k\right), d f\left(e_j\right)\right)b_{ij}dv \\
			& +\quad\int_M (m-1)K\sum_{k=1}^m h\left(\left(d f\left(e_i\right), d f\left(e_k\right)\right) d f\left(e_k\right), d f\left(e_j\right)\right)b_{ij}\\
			& -\sum_{i, j, k=1}^m h\left( R^N\left(d f\left(e_i\right), d f\left(e_k\right)\right) d f\left(e_k\right), d f\left(e_j\right)\right) h\left(d f\left(e_i\right), d f\left(e_j\right)\right) dv\\
			&\geq \int_M (R_{ik}-(m-1)Kb_{ik})b_{kj}b_{ij}+((m-1)K-a)b_{ik}b_{kj}b_{ij}+b_{ij}^2b_{kk} dv_g
		\end{split}
	\end{equation}

We can conclude from the above inequality that   $\nabla\left(f^* h\right)=0$ and $df=0.$

%
	\end{proof}

	\bibliographystyle{acm}	
	\bibliography{F:/2025laterusedbibfile/frommrefandmrlookups,D:/myonlybib/myonlymathscinetbibfrom2023, D:/myonlybib/low-quality-bib-to-publish, C:/Users/Administrator/Desktop/mybib2023}
\end{document}